\documentclass{amsart}
\usepackage[utf8]{inputenc}
\usepackage{amsmath}
\usepackage{amssymb}
\usepackage{amsthm}
\usepackage{color}
\usepackage{graphicx}
\usepackage{subeqnarray}
\usepackage[active]{srcltx}

\newtheorem{theorem}{\bf Theorem}[section]
\newtheorem{lem}{\bf Lemma}[section]
\newtheorem{rem}{\bf Remark}[section]
\newcommand{\R}{\mathbb{R}}
\makeatletter \renewcommand\@biblabel[1]{#1.} \makeatother

\makeatletter \renewcommand\@biblabel[1]{#1.} \makeatother

\begin{document}

\title{Poiseuille Flow of Carreau-Yasuda Fluid at Variable Pressure Gradient} 
\author[N. Kutev]{N. Kutev}
\address{Nikolay Kutev, Institute of Mathematics and Informatics, Bulgarian Academy of Sciences, Acad. G. Bontchev str., bl. 9, 1113 Sofia, Bulgaria}
\email{kutev@math.bas.bg}

\author[S. Tabakova]{S. Tabakova*}
\address{Sonia Tabakova, Institute of Mechanics, Bulgarian Academy of Sciences, Acad. G. Bontchev str., bl. 4, 1113 Sofia, Bulgaria}
\email{stabakova@gmail.com}
\thanks{*corresponding author}



\date{}

\begin{abstract}The unsteady Poiseuille flow of Carreau-Yasuda fluid in a pipe, caused by a variable pressure gradient, is studied theoretically. In a particular case, the steady flow is considered separately. It is proved that at some values of the viscosity model parameters, the problem has a generalized solution, while at others -  classical solution. For the latter, a necessary and sufficient condition is found, which depends on the maximum pressure gradient and on the Carreau-Yasuda model parameters.
\end{abstract}

\maketitle

\section{Introduction}\label{sec1}
Only a small group of fluids refer to the so-called Newtonian fluids, possessing constant viscosity. All other fluids are known as non-Newtonian fluids, whose properties are complicated and usually described by different nonlinear rheological models for shear stress or viscosity \cite{Bird:1987}, \cite{Ewoldt:2022}. Their rheological complexity permits them to be used in a large range of applications,
such as biology, energy, additive manufacturing, etc. 

Usually, the viscosity (or stress) is described by nonlinear models as a function of shear rate \cite{Bird:1987}, such as the power law model, Carreau model, Carreau-Yasuda model and others. These fluids can be shear-thickening, when their viscosity increases with the shear rate (power index $n>1$) or shear thinning in the decreasing case (power index $n<1$). 

However, the complex shear-thinning fluids are often unstable at high shear rates, which results in a negative slope of stress \cite{Divoux:2016}-\cite{Kotsilkova:2023}, usually expressed by $n\lesssim 0$. Then, for the Poiseuille flow in pipes, the pressure axial gradient is no more constant and becomes a function of the radial coordinate, which means that the flow can not be described by the well-known  Weissenberg–Rabinowitsch–Mooney theory \cite{Kim:2018}. 

This paper is a prolongation of our previous works \cite{Kutev:2021}-\cite{Kutev:2023}, which treat the general  flow problems of shear-thinning fluid flow in a pipe. The first two papers concern the unsteady flow case, while the third - the steady case. In these works, the Carreau-Yasuda model is used for the fluid viscosity of the flow, caused by a constant pressure gradient in the axial direction, which is time-dependent or steady. As a result, the flow is governed by a single non-linear PDE of parabolic type in the unsteady case or elliptic type in the steady case for the axial velocity. In \cite{Kutev:2022} it was proved that the unsteady problem becomes uniformly parabolic, nonuniformly parabolic, degenerate parabolic  or backward parabolic at different values of $n$. In \cite{Kutev:2023} the problem has a classical solution, for which a necessary and sufficient condition is found depending on the model parameters. 

In the present paper, the pressure gradient is considered as a function of the radial coordinate $r$ in the pipe flow of the Carreau-Yasuda fluid with arbitrary power index $n$. The existence of a classical solution will be proved separately for the steady and unsteady cases. 

The dimensionless velocity equation of Carreau-Yasuda flow in an infinite circular pipe at axial pressure gradient $b(Y)$ in cylindrical coordinates is given by \cite{Kutev:2022}, \cite{Kutev:2023}: 
\begin{eqnarray}\label{1.1}
8\beta^2U_T - \frac{1}{Y}\frac{\partial}{\partial Y} \left [\left(1+\kappa^\alpha \mid U_Y \mid  ^\alpha\right)^{\frac{n-1}{\alpha}} Y U_{Y}\right]=b(Y),
\\ \nonumber
\textrm{in} \quad Q=\{(T,Y); \quad T>0; \quad Y\in(0, R)\},
\end{eqnarray}
\begin{equation}\label{1.2}
U_Y(T,0) = U(T,R) = 0 \quad \textrm{for} \quad T \geq 0,  \quad
U(0,Y)=\Psi(Y) \quad \textrm{for} \quad  Y\in[0, R],
\end{equation}
where $U=U(T,Y)$ is the dimensionless axial velocity, $T$ - dimensionless time, $Y$ - dimensionless radial coorinate, $R$ - dimensionless radius, $\displaystyle{8\beta^2= Re St}$ with $Re$ and $St$ as Reynolds and Strouhal numbers, $\kappa$ - Carreau number (Weissenberg number), $\alpha$ and $n$ are empirically determined. The function $\Psi(Y)\in C^4([0,R])$ satisfies the compatibility conditions:
\begin{equation}\label{1.3}
\Psi'(0) = \Psi(R) = 0,\quad \Psi'(R) = 0, \quad \Psi''(R) +b(R) = 0
  \end{equation}
\begin{eqnarray}  \label{1.4}
\textrm{and} \quad b(Y)=C^2 ([0,R]), \quad 0\le b(Y) \le b_0 \quad \textrm{for}\quad Y\in [0, R],
\\ \nonumber
 b_0=const.,\quad b(Y)\not\equiv\ 0 \quad \textrm{for} \quad Y\in(0, R).
   \end{eqnarray}
   
In non-divergence form eq. \eqref{1.1} becomes :
\begin{equation}\label{1.5}
P_0(U)=8\beta^2U_T - \Phi(\mid U_Y \mid)U_{YY} -\frac{1}{Y}\left(1+\kappa^\alpha \mid U_Y \mid ^\alpha\right)^{\frac{n-1}{\alpha}}U_{Y} = b(Y),
\end{equation}
where
\begin{eqnarray}\label{1.6}
\Phi(\eta) = (1-n)\left(1+\kappa^\alpha\eta^\alpha\right)^{\frac{n-1-\alpha}{\alpha}} + n\left(1+\kappa^\alpha\eta^\alpha\right)^{\frac{n-1}{\alpha}}
\\ \nonumber
=\left(1+n\kappa^\alpha\eta^\alpha\right)\left(1+\kappa^\alpha\eta^\alpha\right)^{\frac{n-1-\alpha}{\alpha}} \quad \textrm{for} \quad \eta \ge 0.
\end{eqnarray}
Since 
\begin{equation}\label{1.7}
\Phi'(\eta) = (n-1)\kappa^\alpha \left(1+\kappa^\alpha\eta^\alpha\right)^{\frac{n-1-2\alpha}{\alpha}}\left(1+\alpha+n\kappa^\alpha\eta^\alpha\right)\eta^{\alpha-1},
\end{equation}
\begin{eqnarray}\label{1.8}
\Psi(0)=1,\quad \lim_{\eta\rightarrow \infty} \Psi(\eta)=\infty, \quad \textrm{for} \quad n>1, \quad \alpha > 0, \quad \kappa \neq 0, \\ \nonumber
\lim_{\eta\rightarrow \infty} \Psi(\eta)=0, \quad \textrm{for} \quad n<1, \quad \alpha > 0, \quad \kappa \neq 0
\end{eqnarray}
it follows that for $\alpha > 0$ , $\eta \ge 0$:
\begin{eqnarray}\nonumber
\textrm{(i)} \quad \Psi(\eta) \ge 1 \quad \textrm{when} \quad n>1, \quad \kappa \neq 0;
\\ \label{1.9}
\textrm{(ii)} \quad \Psi(\eta) \equiv 1 \quad \textrm{when} \quad n=1 \quad \textrm{or} \quad \kappa = 0;
\\ \nonumber
\textrm{(iii)} \quad \Psi(\eta) \in (0,1] \quad \textrm{when} \quad n\in [0,1), \quad \kappa \neq 0.
\end{eqnarray}
When $n<0$, $\kappa\neq 0$, $\alpha>0$ and $\displaystyle{\eta_0=\kappa^{-1}\left(-\frac{1}{n}\right)^{\frac{1}{\alpha}}}$, then
\begin{align}\nonumber
\textrm{(i)} \quad &\Psi(\eta) \in [0,1] \quad  \textrm{for} \quad \eta \in [0,\eta_0] \quad \textrm{and} 
\\ \label{1.10}
 \textrm{(ii)} \quad &\Psi(\eta)< 0 \quad \textrm{for} \quad \eta>\eta_0.
\end{align}
Thus equation \eqref{1.1} becomes for $\alpha>0$, $\beta>0$:
\begin{align}\label{1.11}
\textrm{(i)} &\quad n>1, \quad \kappa\neq 0 \quad \textrm{- singular, strictly nonuniformly quasilinear parabolic one};
\\ \nonumber
\textrm{(ii)} &\quad n=1  \quad \textrm{or} \quad \kappa=0 \quad \textrm{- singular, linear parabolic one};
\\ \nonumber
\textrm{(iii)} &\quad n \in [0,1) \quad \kappa \neq 0 \quad \textrm{- singular, degenerate at infinity, quasilinear parabolic one}.
\end{align}
When $n<0$, $\alpha>0$, $\beta>0$ and $\kappa\neq 0$, the structure of \eqref{1.1} is more complicated. If
\begin{equation}\label{1.12}
0\leq \mid U_Y(T,Y)\mid < \eta_0 =\kappa^{-1}\left(-\frac{1}{n}\right)^{\frac{1}{\alpha}}
\end{equation}
then \eqref{1.1} is singular, quasilinear uniformly parabolic one, while for 
\begin{equation}\label{1.13}
\mid U_Y(T,Y)\mid \ge \eta_0 
\end{equation}
is singular, degenerate backward quasilinear parabolic one.

Further on we consider the regularized problem
\begin{align} \nonumber 
P_\varepsilon U^\varepsilon=8\beta^2U^\varepsilon_T - \frac{1}{Y+\varepsilon}\frac{\partial}{\partial Y} \left [\left(1+\kappa^\alpha \mid U^\varepsilon_Y \mid  ^\alpha\right)^{\frac{n-1}{\alpha}} (Y+\varepsilon) U^\varepsilon_{Y}\right]=b(Y) \quad \textrm{in} \quad Q \\ \label{1.14}
U^\varepsilon_Y(T,0) = U^\varepsilon(T,R) = 0 \quad \textrm{for} \quad T \geq 0,  \quad
U^\varepsilon(0,Y)=\Psi(Y) \quad \textrm{for} \quad  Y\in[0, R]
\end{align}
for every sufficiently small positive $\varepsilon \in (0,\varepsilon_0]$, where $\varepsilon_0 \ll R$.
\section{Steady Poiseuille flow of Carreau-Yasuda fluid}\label{sec2}
In this section we prove necessary and sufficient conditions for existence and uniqueness of classical solution of the stationary part of equation \eqref{1.1}. For convenience we consider the regularized problem:
\begin{align}\label{2.1}
L V^\varepsilon=\frac{1}{Y+\varepsilon}\frac{\partial}{\partial Y} \left [\left(1+\kappa^\alpha \mid V^\varepsilon_Y \mid  ^\alpha\right)^{\frac{n-1}{\alpha}} (Y+\varepsilon) V^\varepsilon_{Y}\right]=b(Y), \quad Y\in (0,R)
\\ \nonumber
V_Y^\varepsilon (0) = 0, \quad V^\varepsilon(R) = 0
\end{align}
for every $\varepsilon \in (0,\varepsilon_0]$, sufficiently small positive $\varepsilon_0 \ll R$.

The solutions of \eqref{2.1} are crucial for the gradient estimate of the solutions of \eqref{1.14} with constants independent of $\varepsilon$. If
\begin{equation}\label{2.2}
B_\varepsilon(Y)=\frac{1}{Y+\varepsilon}\int^Y_0(s+\varepsilon)b(s)ds \quad \textrm{for} \quad  Y\in[0, R]
\end{equation}
then from the l'Hopital rule for $\varepsilon=0$, we get
\begin{equation}\label{2.3}
\lim_{Y\rightarrow 0} B_0(Y)=\lim_{Y\rightarrow 0} Y b(Y)=0, \quad B_\varepsilon(0)=0 \quad \textrm{for} \quad \varepsilon \in (0,\varepsilon_0], \quad \varepsilon_0 \ll R,
\end{equation}
\begin{equation}\label{2.4}
B'_\varepsilon(Y)=b(Y)-\frac{1}{(Y+\varepsilon)^2}\int^Y_0(s+\varepsilon)b(s)ds,
\end{equation}
\begin{eqnarray} \nonumber
\lim_{Y\rightarrow 0} B'_0(Y)=b(0)-\lim_{Y\rightarrow 0}\frac{Y.b(Y)}{2Y}=\frac{1}{2}b_0(0)\\ \nonumber
B'_\varepsilon (0)=b(0), \quad \textrm{for} \quad \quad \varepsilon \in (0,\varepsilon_0],
\\ \nonumber
B''_\varepsilon(Y)=b'(Y)-\frac{1}{Y+\varepsilon}b(Y)+\frac{2}{(Y+\varepsilon)^3}\int^Y_0(s+\varepsilon)b(s)ds\\ \nonumber
\textrm{and} \quad B_\varepsilon(Y) \in C^2([0,R]) \quad \textrm{for} \quad \varepsilon \in (0,\varepsilon_0], \quad \varepsilon_0 \ll R,\\ \nonumber
B_0(Y)\in C^1([0,R])\cap C^2((0,R])
\end{eqnarray}
We also define the function
\begin{equation}\label{2.5}
F(\eta)=\left(1+\kappa^\alpha\mid\eta \mid^\alpha \right)^\frac{n-1}{\alpha} \eta \quad \textrm{for} \quad \eta \in \R.
\end{equation}
\begin{theorem} \label{thm1}
Suppose $n>0$,$\kappa\ne 0$ or $n \in \R$, $\kappa=0$ and $\alpha>0$. Then problem \eqref{2.1} has a unique classical solution $V^\varepsilon(Y)\in C^2([0,R])$
\begin{equation}\label{2.6}
V^\varepsilon(Y)=-\int_Y^R F^{-1}(B_\varepsilon(s))ds \quad \textrm{for} \quad Y\in[0,R]
\end{equation}
and the estimate
\begin{equation}\label{2.7}
0\le V_Y^\varepsilon(Y)\le F^{-1}(B_{\varepsilon_0}(Y))\le  F^{-1}\left(\frac{b_0(R+\varepsilon_0)}{2}\right)
\end{equation}
holds for every $Y\in [0,R]$ and $\varepsilon \in (0,\varepsilon_0], \quad \varepsilon_0 \ll R$.
\end{theorem}
\begin{proof}
Integrating \eqref{2.1} we get the identity
\begin{equation}\label{2.8}
F( V_Y^\varepsilon)=\left(1+\kappa^\alpha \mid V^\varepsilon_Y \mid  ^\alpha\right)^{\frac{n-1}{\alpha}} V_Y^\varepsilon=B_\varepsilon(Y) 
\end{equation}
for every $Y\in [0,R]$ and $\varepsilon \in (0,\varepsilon_0], \quad \varepsilon_0 \ll R$.

For the function $F(\eta)$ defined in \eqref{2.5} we get the identities
\begin{equation}\label{2.9}
F'(\eta)=\left(1+\kappa^\alpha \mid \eta \mid  ^\alpha\right)^{\frac{n-1-\alpha}{\alpha}}\left(1+n\kappa^\alpha \mid \eta \mid  ^\alpha\right) \end{equation}
\begin{equation}\label{2.10}
F''(\eta)=(n-1)\kappa^\alpha\left(1+\kappa^\alpha \mid \eta \mid  ^\alpha\right)^{\frac{n-1-2\alpha}{\alpha}}\mid \eta \mid  ^{\alpha-2} \eta\left(1+\alpha+n\kappa^\alpha \mid \eta \mid  ^\alpha\right)\quad \textrm{for} \quad \eta \in \R. 
\end{equation}
Since $F'(\eta)>0$ for every $\eta \in \R$, it follows that $F(\eta)$ is strictly monotone increasing function $F(\eta):[0,\infty)\rightarrow [0,\infty)$ because $F(0)=0$, $\lim_{\eta\rightarrow\infty}F(\eta)=\infty$. Hence, there exits the inverse function $F^{-1}(\zeta):[0,\infty)\rightarrow [0,\infty)$ and from \eqref{2.8}, \eqref{2.1} we get
\begin{equation}\nonumber
V^\varepsilon(Y)=-\int_Y^R F^{-1}(B_\varepsilon(s))ds \quad \textrm{for} \quad Y\in[0,R].
\end{equation}
Since 
\begin{equation}\nonumber
\frac{\partial}{\partial\varepsilon}B_\varepsilon(Y)=\frac{1}{(Y+\varepsilon)^2}\int_0^Y \frac{Y-s}{Y+\varepsilon}b(s)ds \ge 0 \quad \textrm{for} \quad \varepsilon \in (0,\varepsilon_0], \quad \varepsilon_0 \ll R
\end{equation}
we have from \eqref{1.4} the estimate
\begin{equation}\label{2.11}
B_\varepsilon (Y) \le B_{\varepsilon_0}(Y)=\frac{1}{Y+\varepsilon_0}\int_0^Y (s+\varepsilon_0)b(s)ds\le \frac{b_0}{Y+\varepsilon_0}\int_0^Y (s+\varepsilon_0)ds \le \frac{b_0}{2}(R+\varepsilon_0)
\end{equation}
and from the monotonicity of $F^{-1}(\zeta)$ we get
\begin{equation}\nonumber
0\le  F^{-1}(B_\varepsilon(Y))\le  F^{-1}\left(\frac{b_0(R+\varepsilon_0)}{2}\right)
\end{equation}
which proves \eqref{2.7}.
\end{proof}
\begin{rem} 
If $n=1$ or $\kappa=0$ then $F(\eta)=\eta$, $F^{-1}(\zeta)=\zeta$ and from \eqref{2.6} it follows that
\begin{eqnarray}
\nonumber
V^\varepsilon(Y)=-\int_Y^R \frac{1}{s+\varepsilon}\int_0^s(t+\varepsilon)b(t)dtds \\ \nonumber
\quad \textrm{for every} \quad Y\in[0,R] \quad \textrm{and} \quad \varepsilon \in (0,\varepsilon_0], \quad \varepsilon_0 \ll R.
\end{eqnarray}
\end{rem}
\begin{theorem} \label{thm2}
Suppose $\alpha> 0$, $n\le 0$, $\varepsilon \in (0,\varepsilon_0], \quad \varepsilon_0 \ll R$. Then problem \eqref{2.1} has a unique classical solution 
$V^\varepsilon(Y)\in C^2([0,R))\cap C^1([0,R])$
\begin{equation}\label{2.12}
V^\varepsilon(Y)=-\int_Y^R F^{-1}(B_\varepsilon(s))ds \quad \textrm{for} \quad Y\in[0,R]
\end{equation}
(i) for $n=0$ iff
\begin{equation}\label{2.13}
B_\varepsilon(Y))\le \kappa^{-1}=\lim_{\eta\rightarrow \infty}F(\eta);
\end{equation}
(ii) for $n<0$ iff
\begin{equation}\label{2.14}
B_\varepsilon(Y))\le \left(\frac{n-1}{n}\right)^\frac{n-1}{\alpha}\kappa^{-1}\left(-\frac{1}{n}\right)^\frac{1}{\alpha}=F\left(\kappa^{-1}\left(-\frac{1}{n}\right)^\frac{1}{\alpha}\right)
\end{equation}
for every $Y\in [0,R]$.

Moreover, the estimate \eqref{2.7} holds.
\end{theorem} 
\begin{rem}
If
\begin{eqnarray}
\nonumber
\textrm{(i)} \quad n=0 \quad \textrm{and} \quad B_\varepsilon(Y_1)=\kappa^{-1}, \quad Y_1\in (0,R) \quad \textrm{then}\\  \nonumber
V^\varepsilon(Y)\in C^2([0,R]\setminus \left\lbrace Y_1 \right\rbrace ) \cap  C^1([0,R]);\\ \nonumber
\textrm{(ii)} \quad n<0 \quad \textrm{and} \quad B_\varepsilon(Y_2)=\left(\frac{n-1}{n}\right)^\frac{n-1}{\alpha}\kappa^{-1}\left(-\frac{1}{n}\right)^\frac{1}{\alpha}, \quad Y_2\in (0,R) \quad \textrm{then}\\ \nonumber
V^\varepsilon(Y)\in C^2([0,R]\setminus\left\lbrace Y_2 \right\rbrace ) \cap  C^1([0,R])
\end{eqnarray}
and $V^\varepsilon(Y)$ is a $C^1([0,R])$ generalized solution of \eqref{2.1}.
\end{rem} 
\begin{proof}[Proof of Theorem~\ref{thm2}]
(i) From \eqref{2.9} we have $\displaystyle{F('\eta)=\left(1+\kappa^\alpha \eta ^\alpha\right)^{-\frac{1+\alpha}{\alpha}}> 0}$ for $\eta \ge 0$ and $F(0)=0$, $\lim_{\eta\rightarrow\infty} F(\eta)=\kappa^{-1}$. Hence $F(\eta)$ is strictly increasing function,
\begin{equation}\label{2.15}
F(\eta):[0,\infty)\rightarrow [0,\kappa^{-1})
\end{equation}
and there exists the inverse function
\begin{equation}\label{2.16}
(F^{-1})(\zeta):[0,\kappa^{-1})\rightarrow [0,\infty)
\end{equation}
Since
\begin{equation}\label{2.17}
(F^{-1})'(\zeta)=\frac{1}{F'(F^{-1}(\zeta))}>0
\end{equation}
the inverse function $F^{-1}(\zeta)$ is strictly monotone increasing. \\
\textbf{Sufficiency:} If \eqref{2.13} holds, then from \eqref{2.8}, \eqref{2.16} and \eqref{2.13}, problem \eqref{2.1} is equivalent to \eqref{2.6}. After integration of  \eqref{2.6} from the boundary condition $V^\varepsilon(R)=0$, we get \eqref{2.12}.\\
\textbf{Necessity:} If $V^\varepsilon (Y)\in  C^1([0,R])\cap C^2([0,R))$ is a classical solution of \eqref{2.1}, we suppose by contradiction that \eqref{2.13} fails, i.e., there exists $Y_0\in [0,R]$ such that
\begin{equation}\label{2.18}
B_\varepsilon(Y_0)> \kappa^{-1} \quad \textrm{for some} \quad \varepsilon \in (0,\varepsilon_0], \quad \varepsilon_0 \ll R.
\end{equation}
From \eqref{2.1}, \eqref{2.16} and \eqref{2.18} at the point $Y_0$, we get the following impossible chain of inequalities
\begin{equation}
\kappa^{-1}\ge \sup_{Y\in [0, Y]}F(V_Y^\varepsilon)=B_\varepsilon(Y_0)>\kappa^{-1},
\end{equation}
which proves the necessity of \eqref{2.13}.\\
(ii) From \eqref{2.9} for $n<0$ it follows that 
\begin{equation}\label{2.19}
F'(\eta)>0 \quad \textrm{for} \quad 0\le \eta <\kappa^{-1}\left(-\frac{1}{n}\right)^\frac{1}{\alpha}, \quad F'(\eta_0)=0,
\end{equation}
\begin{equation}\nonumber
\eta_0=\kappa^{-1}\left(-\frac{1}{n}\right)^\frac{1}{\alpha} \quad \textrm{and} \quad F'(\eta)<0 \quad \textrm{for} \quad  \eta>\kappa^{-1}\left(-\frac{1}{n}\right)^\frac{1}{\alpha}
\end{equation}
\begin{equation}\label{2.20}
F(0)=0, \quad F(\eta_0)=\left(\frac{n-1}{n}\right)^\frac{n-1}{\alpha}\kappa^{-1}\left(-\frac{1}{n}\right)^\frac{1}{\alpha}, \quad \lim_{\eta\rightarrow\infty}F(\eta)=0.
\end{equation}
Hence the function
\begin{equation}\label{2.21}
F(\eta):[0,\eta_0)\rightarrow [0,F(\eta_0))
\end{equation}
is strictly monotone increasing one, while
\begin{equation}\label{2.22}
F(\eta):[\eta_0,\infty)\rightarrow [F(\eta_0),0)
\end{equation}
is strictly monotone decreasing one.\\
From \eqref{2.21} and \eqref{2.17} the inverse function $F^{-1}(\zeta)$ of $F(\eta)$ in $[0,\eta_0]$ exits
\begin{equation}\label{2.23}
F^{-1}(\zeta):[0,\zeta_0]\rightarrow [0,\eta_0], \quad \zeta_0=F(\eta_0)=\left(\frac{n-1}{n}\right)^\frac{n-1}{\alpha}\kappa^{-1}\left(-\frac{1}{n}\right)^\frac{1}{\alpha}
\end{equation}
and is strictly monotone increasing one,
\begin{equation}\label{2.24}
F^{-1}(\zeta)\in C^2([0,\eta_0))\cap C([0,\eta_0]).
\end{equation}
\textbf{Sufficiency:} If \eqref{2.13} holds then $F^{-1}(B_\varepsilon(Y))$ is well defined for $Y\in [0,R]$ and integrating \eqref{2.8}, we get \eqref{2.12}. From \eqref{2.24} it follows that $V^\varepsilon (Y)\in  C^2([0,R))\cap C^1([0,R])$.\\
\textbf{Necessity:} If $V^\varepsilon (Y)\in  C^2([0,R))\cap C^1([0,R])$ is a classical solution of \eqref{2.1}, we suppose by contradiction that \eqref{2.14} fails, i.e., there exists $Y_0\in [0,R]$ such that 
\begin{equation}\label{2.25}
B_\varepsilon(Y_0)>\left(\frac{n-1}{n}\right)^\frac{n-1}{\alpha}\kappa^{-1}\left(-\frac{1}{n}\right)^\frac{1}{\alpha}
\end{equation}
for some $\varepsilon \in (0,\varepsilon_0], \quad \varepsilon_0 \ll R$. From the continuity of $B_\varepsilon(Y)$, without loss of generality, we assume that $Y_0\in (0,R)$. From \eqref{2.8} at the point $Y_0$ and \eqref{2.20}, \eqref{2.21}, \eqref{2.25} we get the following impossible chain of inequalities
\begin{eqnarray}
\nonumber
\left(\frac{n-1}{n}\right)^\frac{n-1}{\alpha}\kappa^{-1}\left(-\frac{1}{n}\right)^\frac{1}{\alpha}\ge \sup_{Y\in[0,R]}F(V_Y^\varepsilon (Y) \ge F(V_Y^\varepsilon (Y_0)\\ \nonumber
=B_\varepsilon(Y_0)>\left(\frac{n-1}{n}\right)^\frac{n-1}{\alpha}\kappa^{-1}\left(-\frac{1}{n}\right)^\frac{1}{\alpha}.
\end{eqnarray}
 The estimate \eqref{2.7} follows from \eqref{2.11} and the monotonicity of $F^{-1}(\zeta)$.
\end{proof}

\section{Unsteady Poisseuille flow of Carreau-Yasuda fluid\protect}\label{sec3}
In this section we formulate and prove the main results in this paper for \eqref{1.1}-\eqref{1.4} for different values of the parameters $\alpha>0$, $\beta>0$,$n\in \R$, $\kappa \geq 0$.

For this purpose we consider the regularized problem \eqref{1.14} in nondimensionless form:
\begin{align}\label{3.1}
P_\varepsilon U^\varepsilon=8\beta^2U^\varepsilon_T -\Phi\left(\mid U_Y^\varepsilon \mid\right)U^\varepsilon_{YY}-\frac{1}{Y+\varepsilon}\left(1+\kappa^\alpha \mid U^\varepsilon_Y \mid  ^\alpha\right)^{\frac{n-1}{\alpha}} U^\varepsilon_{Y}=b(Y) \quad \textrm{in} \quad Q \\ \nonumber
U^\varepsilon_Y(T,0) = U^\varepsilon(T,R) = 0 \quad \textrm{for} \quad T \geq 0,  \quad
U^\varepsilon(0,Y)=\Psi(Y) \quad \textrm{for} \quad  Y\in[0, R] \\ \nonumber \textrm{and} \quad \varepsilon \in (0,\varepsilon_0], \quad \varepsilon_0 \ll R.
\end{align}

In the following lemmas we prove apriori estimated for $\mid U^\varepsilon (T,Y) \mid$ and $\mid U_Y^\varepsilon (T,Y) \mid$ in $\overline{Q}$, $Q=\left\lbrace (T,Y); T>0, Y\in (0,R)\right\rbrace$ with constants independent of $\varepsilon$.
\begin{lem}
Suppose $\mid U^\varepsilon (T,Y) \mid \cap C^2(Q) \cap C^1(\overline{Q})$ is a solution of \eqref{3.1}, $\alpha>0$, $\beta>0$ and either $n \ge 1$, $\kappa \neq 0$, or $\kappa=0$, $n\in \R$. Then the estimates 
\begin{equation}\label{3.2}
\mid U^\varepsilon (T,Y) \mid \le K_1(R^2-Y^2)\le K_1R^2,
\end{equation}
\begin{equation}\label{3.3}
\mid U_Y^\varepsilon (T,R) \mid \le 2K_1R
\end{equation}
hold for $T\ge 0$, $Y\in [0,R]$, where 
\begin{equation}\label{3.4}
K_1=\max \left\lbrace\sup_{Y\in[0,R]}\left \vert \frac{\Psi(Y)}{R^2-Y^2}\right\vert, \frac{1}{2} b_0 \right\rbrace.
\end{equation}
\end{lem}
\begin{proof}
For the function $H(T,Y)=K_1(R^2-Y^2)$ and the operator
\begin{equation}\label{3.5}
PW=8\beta^2W_T -\Phi\left(\mid U_Y^\varepsilon \mid\right)W_{YY}-\frac{1}{Y+\varepsilon}\left(1+\kappa^\alpha \mid U^\varepsilon_Y \mid  ^\alpha\right)^{\frac{n-1}{\alpha}} W_{Y}
\end{equation}
we get from \eqref{3.4}, \eqref{1.9}$_i$ and \eqref{1.9}$_{ii}$ the estimate 
\begin{align}\nonumber
PH=&2\Phi\left(\mid U_Y^\varepsilon \mid\right)K_1+\frac{2K_1Y}{Y+\varepsilon}\left(1+\kappa^\alpha \mid U^\varepsilon_Y \mid  ^\alpha\right)^{\frac{n-1}{\alpha}} \ge 2K_1\ge b_0 \ge b(Y)=PU^\varepsilon \\
&\textrm{for} \quad (T,Y)\in Q.\nonumber
\end{align}
Hence $P(H-U^\varepsilon ) \ge 0$ in $Q$, $H(0,Y)-U^\varepsilon(0,Y)=K_1(R^2-Y^2)-\Psi(Y)\ge 0$ for $Y\in [0,R]$, $H(T,R)-U^\varepsilon(T,R)=0$ and $H_Y(T,0)-U_Y^\varepsilon(T,0)=0$ for $T\ge 0$. From the strong interior maximum principle $H(T,Y)-U^\varepsilon(T,Y)$ does not attain maximum or minimum at some interior point of $Q$ and from the strong boundary maximum principle also on $\Gamma_1=\left\lbrace(T,0); T>0 \right\rbrace$ \cite{Ladyzhenskaya et al}, \cite{Protter_Weinberger:1967}. Hence $H(T,Y)-U^\varepsilon(T,Y)$ attains its maximum and minimum on the rest of the parabolic boundary $\Gamma_2\cup\Gamma_3$:
\begin{equation}\label{3.6}
\Gamma_2=\left\lbrace(0,Y);Y\in[0,R]\right\rbrace, \quad \Gamma_3=\left\lbrace(T,R);T\ge 0\right\rbrace.
\end{equation}
The estimate \eqref{3.2} follows from the choice of $K_1$ and the zero boundary condition on $\Gamma_3$.

From \eqref{3.2} the boundary gradient estimate becomes:
\begin{equation}\label{3.7}
\mid U_Y^\varepsilon(T,R) \mid\le 2K_1R \quad \textrm{for} \quad T\ge 0.
\end{equation}
\end{proof}
\begin{lem}
Suppose $U^\varepsilon (T,Y) \cap C^2(Q) \cap C^1(\overline{Q})$ is a solution of \eqref{3.1}, $\alpha>0$, $\beta>0$ and one of the following conditions holds
\begin{align}  \label{3.8} 
&\textrm{(i)} \quad n \in (0,1), \kappa \neq 0  \quad \textrm{or}
\\ \nonumber
&\textrm{(ii)} \quad n=0, \kappa\neq 0, \sup_{Y\in [0,R]} B_0(Y)< \kappa^{-1}\quad \textrm{or}\\ \nonumber
&\textrm{(iii)} \quad n < 0, \kappa\neq 0, \sup_{Y\in [0,R]} B_0(Y)< \left(\frac{n-1}{n}\right)^{\frac{n-1}{\alpha}} \kappa^{-1}\left(-\frac{1}{n}\right)^{\frac{1}{\alpha}}.
\end{align} 
If
\begin{equation} \label{3.9}
\mid\Psi(Y) \mid < -V^0(Y) \quad \textrm{for} \quad Y\in [0,R], 
\end{equation}
where $V^0(Y)$ is defined in Theorem \ref{thm1} in case \eqref{3.8}$_{(i)}$ and in Theorem \ref{thm2} in case \eqref{3.8}$_{(ii)}$ and \eqref{3.8}$_{(iii)}$, then the estimates
\begin{equation} \label{3.12}
\mid U^\varepsilon (T,Y) \mid \le (R-Y)F^{-1}\left(\frac{b_0(R+\varepsilon_0)}{2} \right\rbrace
\end{equation}
\begin{equation} \label{3.13}
\mid U_Y^\varepsilon (T,R) \mid \le F^{-1}\left(\frac{b_0(R+\varepsilon_0)}{2}\right\rbrace \quad \textrm{hold for} \quad T\ge 0, Y\in [0,R].
\end{equation}
\begin{proof}
If $V^\varepsilon (Y)$ is defined in Theorem \ref{thm1} for \eqref{3.8}$_{(i)}$ and in Theorem \ref{thm2} for \eqref{3.8}$_{(ii)}$ and \eqref{3.8}$_{(iii)}$, then for the operator $P$ given in \eqref{3.5}, we have
\begin{align}
\nonumber
PV^\varepsilon&=-b(Y)-A_1\left[ (U^\varepsilon_Y)^2-(V^\varepsilon_Y)^2\right ] \quad \textrm{in Q, where}\\ \nonumber
A_1&=\frac{1}{2}\kappa^\alpha \int_0^1\left[\theta(U^\varepsilon_Y)^2+(1-\theta)(V^\varepsilon_Y)^2\right]^{\frac{\alpha-2}{2}}d\theta \Bigl\{ n(n-1)\left(V^\varepsilon_{YY}+\frac{V^\varepsilon_{Y}}{Y+\varepsilon}\right )\\ \nonumber
&\int_0^1\left[\theta\left(1+\kappa^\alpha \mid U^\varepsilon_Y \mid  ^\alpha\right)+(1-\theta)\left(1+\kappa^\alpha \mid V^\varepsilon_Y \mid  ^\alpha\right)\right]^{\frac{n-1-\alpha}{\alpha}}d\theta\\ \nonumber
&+V^\varepsilon_{YY}(1-n)(n-1-\alpha)\int_0^1\left[\theta\left(1+\kappa^\alpha \mid U^\varepsilon_Y \mid  ^\alpha\right)+(1-\theta)\left(1+\kappa^\alpha \mid V^\varepsilon_Y \mid  ^\alpha\right)\right]^{\frac{n-1-2\alpha}{\alpha}}d\theta \Bigr\}
\end{align}
In the above calculations, we use the identity
\begin{align}\nonumber
\frac{V^\varepsilon_{Y}}{Y+\varepsilon} \left[\left(1+\kappa^\alpha \mid U^\varepsilon_Y \mid  ^\alpha\right)^{\frac{n-1}{\alpha}}-\left(1+\kappa^\alpha \mid V^\varepsilon_Y \mid  ^\alpha\right)^{\frac{n-1}{\alpha}}\right]&+\left[\Phi(\mid U^\varepsilon_Y \mid)-\Phi(\mid V^\varepsilon_Y \mid)\right]V^\varepsilon_{YY}\\ \nonumber
&=-A_1\left[(U^\varepsilon)^2-(V^\varepsilon)^2\right]
\end{align}
Thus the function $W=U^\varepsilon(T,Y)+V^\varepsilon(T,Y)$ satisfies the problem
\begin{align} \label{3.14} 
&PW-A_1\left(U^\varepsilon_Y-V^\varepsilon_Y\right)W_Y=0 \quad \textrm{in Q}\\ \nonumber
&W_Y(T,0)=W(T,R)=0 \quad \textrm{for} \quad T\ge 0, W(0,Y)=\Psi(Y)+V^\varepsilon(Y)\le 0 \quad \textrm{for} \quad Y\in [0,R] \\ \nonumber &\textrm{and} \quad \varepsilon \in (0,\varepsilon_0], \varepsilon_0 \ll R.
\end{align}
Hence, from the strong interior and boundary maximum principle for classical solutions \cite{Ladyzhenskaya et al}, \cite{Protter_Weinberger:1967}, it follows that $W(T,Y)$ does not attain a positive maximum in $\overline{Q}$, i.e., 
\begin{equation}\nonumber
W(T,Y)\le 0 \quad \textrm{in } \overline{Q} \quad \textrm{and} \quad U^\varepsilon(T,Y)\le -V^\varepsilon(T,Y)\quad \textrm{for} \quad T\ge 0, Y\in[0,R].
\end{equation}
The opposite inequality $U^\varepsilon(T,Y)\ge V^\varepsilon(T,Y)$ follows in the same way by means of the function $W_1=-U^\varepsilon(T,Y)+V^\varepsilon(T,Y)$, which satisfies \eqref{3.14} with $W_1(0,Y)=-\Psi(Y)+V^\varepsilon(Y)\le 0$ for $Y\in[0,R]$ and sufficiently small positive $\varepsilon$. Thus \eqref{3.12} follows from \eqref{2.14} and the monotonicity of $F^{-1}(\zeta)$. The estimate \eqref{3.13} is a trivial corollary of \eqref{3.12}.
\end{proof}
\end{lem}
\begin{lem}
Suppose $\mid U^\varepsilon (T,Y) \mid \cap C^3(Q) \cap C^2(\overline{Q})$ is a solution of \eqref{3.1}, $\alpha>0$, $\beta>0$, $n\in \R$, $\kappa\ge 0$. Then the estimate
\begin{equation}\label{3.15}
\mid U_T^\varepsilon (T,Y) \mid \le K_2 \quad \textrm{for} \quad T\ge0, Y\in[0,R], \varepsilon \in (0,\varepsilon_0], \varepsilon_0 \ll R,
\end{equation}
holds, where
\begin{equation}\label{3.16}
K_2=\frac{1}{8\beta^2}\Bigl[\sup_{Y\in[0,R]}\vert \Phi\left(\vert \Psi'(Y)\vert \right)\Psi''(Y)\vert +\sup_{Y\in[0,R]}\left(1+\kappa^\alpha\vert \Psi'(Y)\vert ^\alpha \right)^\frac{n-1}{\alpha}\Bigl\vert \frac{\Psi'(Y)}{Y} \Bigr\vert +b_0 \Bigr]
\end{equation}
\end{lem}
\begin{proof}
Differentiating \eqref{3.1} with respect to $T$, we obtain that $U_T^\varepsilon$ satisfies the problem
\begin{align} \label{3.17} 
P_3U_T^\varepsilon=0 \quad \textrm{in Q}, \quad U_{TY}^\varepsilon(T,0)=0, \quad U_T^\varepsilon(T,R)=0 \quad \textrm{for} \quad T\ge 0\\ 
\nonumber
U_T^\varepsilon(0,Y)=\frac{1}{8\beta^2}\Bigl[ \Phi\left(\vert \Psi'(Y)\vert \right)\Psi''(Y)+\frac{1}{Y+\varepsilon}\left(1+\kappa^\alpha\vert \Psi'(Y)\vert ^\alpha \right)^\frac{n-1}{\alpha}\Psi'(Y)  +b(Y) \Bigr],
\end{align}
where
\begin{align}\nonumber
P_3W&=8\beta^2W_T-\Phi(\vert U_Y^\varepsilon \vert)W_{YY}-\Bigl[\frac{1}{Y+\varepsilon}\left(1+\kappa^\alpha \mid U^\varepsilon_Y \mid  ^\alpha\right)^{\frac{n-1-\alpha}{\alpha}}\\
\nonumber
&\left(1+\kappa^\alpha \mid U^\varepsilon_Y \mid ^\alpha+(n-1)\alpha\kappa^\alpha  U_Y^\varepsilon\mid U^\varepsilon_Y \mid ^{\alpha-2}\right)
+\kappa^\alpha U_Y^\varepsilon\mid U^\varepsilon_Y \mid ^{\alpha-2}U^\varepsilon_{YY}\\
\nonumber
&\left(1+\kappa^\alpha \mid U^\varepsilon_Y \mid  ^\alpha \right)^{\frac{n-1-2\alpha}{\alpha}}\left( 2n-1-\alpha+n(n-\alpha)\kappa^\alpha \mid U^\varepsilon_Y \mid ^\alpha \right)\Bigr]W_Y
\end{align}
From the strong interior and boundary maximum principle \cite{Ladyzhenskaya et al}, \cite{Protter_Weinberger:1967} $U^\varepsilon_T (T,Y)$ attains its maximum and minimum in $\overline{Q}$ on the part $\Gamma_2\cup\Gamma_3$ of the parabolic boundary. Estimate \eqref{3.15} holds from the zero boundary conditions on $\Gamma_3$ and the choice of $K_3$.
\end{proof}
\begin{lem}
Under the assumptions of Lemma 3.3 the estimate 
\begin{equation}\label{3.18}
\left(1+\kappa^\alpha \mid U^\varepsilon_Y \mid  ^\alpha\right)^{\frac{n-1}{\alpha}}\mid U^\varepsilon_Y \mid  ^\alpha\le K_3(Y+\varepsilon) \quad \textrm{for} \quad T\ge0, Y\in[0,R], \varepsilon \in (0,\varepsilon_0], \varepsilon_0 \ll R,
\end{equation}
holds, where
 \begin{equation}\label{3.19}
 K_3=\frac{1}{2}b_0+4\beta^2K_2
 \end{equation}
\end{lem}
\begin{proof}
Integrating \eqref{1.14} from $0$ to $Y\in (0,R]$ we get  from \eqref{1.5}, \eqref{3.15} the estimates
\begin{eqnarray}\nonumber
(Y+\varepsilon)\left(1+\kappa^\alpha \mid U^\varepsilon_Y \mid  ^\alpha\right)^{\frac{n-1}{\alpha}}\mid U^\varepsilon_Y \mid =\Bigl\vert\int_0^Y(s+\varepsilon)\left[8\beta^2U_T^\varepsilon(T,s)-b(s)\right]ds \Bigr\vert\\ \nonumber
\le \left(8\beta^2K_2+b_0\right)\int_0^Y(s+\varepsilon)ds\le K_3(Y+\varepsilon)^2
\end{eqnarray}
which proves \eqref{3.8}.
\end{proof}
\begin{lem}
Suppose $U^\varepsilon (T,Y) \cap C^3(Q) \cap C^2(\overline{Q})$ is a solution of \eqref{3.1}, $\alpha>0$, $\beta>0$\\

(i) If $n>0$, $\kappa\neq 0$ or $\kappa=0$, $n\in\R$ then the estimate
\begin{align}\label{3.20}
 \mid U_Y^\varepsilon (T,Y) \mid \le \max \Bigl\{K_3(R+\varepsilon_0), \left[K_3(R+\varepsilon_0)\kappa^{-\alpha}\right]^\frac{1}{n}\Bigr\}=K_4 \quad &\textrm{holds for} \quad (T,Y)\in \overline{Q} 
 \\ & \textrm{and} \quad \varepsilon \in (0,\varepsilon_0], \varepsilon_0 \ll R;\nonumber
 \end{align}
 
 (ii) If \eqref{3.9} and \eqref{3.8}$_{ii}$ are satisfied, then the estimate
 \begin{equation}\label{3.21}
 \mid U_Y^\varepsilon (T,Y) \mid \le K_5\exp(T) \quad \textrm{holds for} \quad (T,Y)\in \overline{Q} \quad \textrm{and} \quad \varepsilon \in (0,\varepsilon_0], \varepsilon_0 \ll R,
 \end{equation}
 where
 \begin{equation}\label{3.22}
K_5=\max \Bigl\{\frac{1}{8\beta^2}\sup_{Y\in[0,R]}\mid b_Y(Y)\mid, \sup_{Y\in[0,R]}\mid\Psi'(Y)\mid, F^{-1}\Bigl(\frac{b_0(R+\varepsilon_0)}{2}\Bigr) \Bigr\}
 \end{equation}
 
 (iii) If \eqref{3.9}, \eqref{3.8}$_{iii}$ and $K_5 < \eta_0$ are satisfied, then the estimate \eqref{3.21} holds in $\overline{Q_\lambda}$,  $Q_\lambda=\lbrace(T,Y); 0<T<\lambda; 0<Y<R\rbrace$, 
 where 
 \begin{equation} \label{3.22*}
 \lambda<\ln\frac{\eta_0}{K_5}, \quad \eta_0=\kappa^{-1}\left(-\frac{1}{n}\right)^{\frac{1}{\alpha}}
 \end{equation}
 
 \end{lem}
  \begin{proof}
 Differentiating \eqref{3.1} with respect to $Y$, we obtain that $U_Y^\varepsilon$ satisfies the problem
 \begin{align}\nonumber
     P_2U_Y^\varepsilon=b_Y(Y) \quad \textrm{in} \quad Q, \quad U_Y^\varepsilon(T,0)=0 \quad \textrm{for} \quad T\ge 0,\\ \label{3.23}
     U_Y^\varepsilon(0,Y)=\Psi'(Y) \quad \textrm{for} \quad Y\in [0,R]
 \end{align}
 where
 \begin{align}\label{3.24}
  P_2W&=8\beta^2W_T-\Phi(\vert U_Y^\varepsilon \vert)W_{YY}-A_2 W_Y +\frac{1}{(Y+\varepsilon)^2}\left(1+\kappa^\alpha \mid U^\varepsilon_Y \mid  ^\alpha\right)^{\frac{n-1}{\alpha}}W,\\ \nonumber
  A_2&=\frac{1}{Y+\varepsilon}\left(1+\kappa^\alpha \mid U^\varepsilon_Y \mid  ^\alpha\right)^{\frac{n-1-\alpha}{\alpha}}\left(1+n\kappa^\alpha \mid U^\varepsilon_Y \mid  ^\alpha\right)\\ 
  \nonumber
  &+(n-1)\kappa^\alpha\left(\alpha+1+n\kappa^\alpha\mid U^\varepsilon_Y \mid  ^\alpha\right)
  U^\varepsilon_Y\mid U^\varepsilon_Y \mid  ^{\alpha-2}\left(1+\kappa^\alpha \mid U^\varepsilon_Y \mid  ^\alpha\right)^{\frac{n-1-2\alpha}{\alpha}}
 \end{align}
 (i) Estimate \eqref{3.20} follows from \eqref{3.18} and the inequalities
  \begin{eqnarray}\nonumber
  \left(1+\kappa^\alpha \mid U^\varepsilon_Y \mid  ^\alpha\right)^{\frac{n-1}{\alpha}}  \mid U^\varepsilon_Y \mid \ge  \mid U^\varepsilon_Y \mid  \quad \textrm{for} \quad n\ge 1, \quad \kappa\ne 0 \quad \textrm{or} \quad \kappa=0 \quad \textrm{and} \quad n\in \R;\\ \nonumber  
  \left(1+\kappa^\alpha \mid U^\varepsilon_Y \mid  ^\alpha\right)^{\frac{n-1}{\alpha}}  \mid U^\varepsilon_Y \mid \ge \left(\kappa^\alpha + \mid U^\varepsilon_Y \mid  ^{-\alpha} \right)\mid U^\varepsilon_Y \mid^n \ge \kappa^\alpha \mid U^\varepsilon_Y \mid  ^n
 \end{eqnarray}
 ii) Under the conditions in Lemma 3.5)$_{ii}$ boundary gradient estimate \eqref{3.13} holds from Lemma 3.2. Simple computations give us 
 \begin{eqnarray}\nonumber
 P_2(K_5\exp(T))=8\beta^2K_5\exp(T)+\frac{K_5}{(Y+\varepsilon)^2} \left(1+\kappa^\alpha \mid U^\varepsilon_Y \mid  ^\alpha\right)^{\frac{n-1}{\alpha}}\exp(T)\ge 8\beta^2K_5\ge b(Y) \quad \textrm{in} \quad Q. 
 \end{eqnarray}
 Thus the function $W(T,Y)=U^\varepsilon_Y (T,Y)-8\beta^2K_5\exp(T)$ is a solution of the problem
\begin{eqnarray}\nonumber
P_2W=b_Y-P_2(K_5\exp(T))\le 0 \quad \textrm{in} \quad Q \\ \nonumber
W(T,0)=-K_5\exp(T)\le 0, \quad W(T,R)=-K_5\exp(T)\le 0  \quad \textrm{for} \quad T\ge 0, \\ \nonumber
W(0,Y)=\Psi'(Y)-K_5\le 0  \quad \textrm{for} \quad Y\in [0,R].
  \end{eqnarray}
From the maximum principle \cite{Ladyzhenskaya et al}, \cite{Protter_Weinberger:1967} we get the estimate $W(T,Y)\le 0$ in $\overline{Q}$, i.e., $U^\varepsilon_Y (T,Y)\le 8\beta^2K_5\exp(T)$  for $T\ge 0$, $Y\in [0,R]$. Analogously, by means of the function $U^\varepsilon_Y (T,Y)+ 8\beta^2K_5\exp(T)$ we obtain the opposite estimate, which proves \eqref{3.21}.

(iii) The proof of (iii) is the same as the proof of (ii) in $\overline{Q_\lambda}$.The only difference is that the operator $P_2$ is uniformly parabolic in $\overline{Q_\lambda}$ and the maximum principle is applicable \cite{Ladyzhenskaya et al}, \cite{Protter_Weinberger:1967}.
 \end{proof}
 \begin{lem}
 Suppose $U^\varepsilon (T,Y) \in C^3(Q) \cap C^1(\overline{Q})$, $\alpha>0$, $\beta>0$. If
 \begin{equation}\label{3.25}
 n>0, \quad \kappa\ne 0, \quad \textrm{or} \quad \kappa=0, \quad n\in \R,
 \end{equation}
 then the estimate
  \begin{equation}\label{3.26}
\mid U_{YY}^\varepsilon (T,Y) \mid \le K_6, \quad \textrm{holds for} \quad T\ge 0, \quad Y\in [0,R] \quad \varepsilon \in (0,\varepsilon_0], \quad \varepsilon_0 \ll R \quad \textrm{with}
 \end{equation}
    \begin{equation}\label{3.27} 
K_6=\Biggl\{
\begin{matrix}
&\left(12\beta^2K_3+\frac{3}{2}b_0\right) \Big/ \Phi(K_4) \quad & \quad\textrm{for} \quad 0<n<1\\
&\left(12\beta^2K_3+\frac{3}{2}b_0\right) &\textrm{for} \quad n>1\\
\end{matrix}
 \end{equation}
 If 
  \begin{equation}\label{3.28}
 n=0, \quad \kappa\ne 0, \quad \sup_{Y\in [0,R]}B_0(Y) < \kappa^{-1}, \quad \textrm{and \eqref{3.9} holds,}
 \end{equation}
 then the estimate 
  \begin{equation}\label{3.30}
\mid U_{YY}^\varepsilon (T,Y) \mid \le \left(12\beta^2K_3+\frac{3}{2}b_0\right)\Big/ \Phi(K_5 \exp(T_0)) =K_7, 
 \end{equation}
is satisfied for $T\in [0,T_0]$, $T_0 < \infty$, $ Y\in [0,R]$, $\varepsilon \in (0,\varepsilon_0]$, $\varepsilon_0 \ll R$.

If
   \begin{equation}\label{3.29}
 n<0, \quad \kappa\ne 0, \quad \sup_{Y\in [0,R]}B_0(Y) < \left(\frac{n-1}{n}\right)^{\frac{n-1}{\alpha}}\kappa^{-1}\left(-\frac{1}{n}\right)^{\frac{1}{\alpha}}, \quad  K_5<\eta_0=\kappa^{-1}\left(-\frac{1}{n}\right)
  \end{equation}
and (49) holds,  then the estimate \eqref{3.30} is obtained for $T\in[0,\lambda]$, $Y\in[0,R]$, where $\lambda$ is defined in \eqref{3.22*}.
  \end{lem}
  \begin{proof}
  Estimates \eqref{3.26}, \eqref{3.30} follow immediately from \eqref{3.1}, \eqref{3.15}, \eqref{3.18}, \eqref{3.20} and \eqref{3.21}.
  \end{proof}
  \begin{lem}
 Suppose $U^\varepsilon (T,Y) \in C^3(Q) \cap C^1(\overline{Q})$ is a solution of \eqref{3.1}, $\alpha>0$, $\beta>0$\\

(i) If \eqref{3.25} holds then the estimate
   \begin{equation}\label{3.31}
\Bigl| \frac{\partial^\gamma}{\partial Y}\Bigl(\frac{\partial^\mu}{\partial T}U^\varepsilon (T,Y)\Bigr)\Bigr| \le K_8, 
 \end{equation} 
 holds for $0\le \gamma + \mu \le 3$, $Y\in [\delta, R]$, $R> \delta > 0$, $T\ge 0$, $\varepsilon \in (0,\varepsilon_0]$, $\varepsilon_0 \ll R$ and $K_8$ depending on $\gamma$, $\mu$, $\delta$, $R$, $K_1$, $K_2$, $K_3$, $K_6$, but is independent of $\varepsilon$;\\
 
(ii) If \eqref{3.28} holds then the estimate \eqref{3.31} is satisfied for  $0\le \gamma + \mu \le 3$, $Y\in [\delta, R]$, $\delta > 0$, $T\in [0,T_0]$, $T_0 < \infty$, $\varepsilon \in (0,\varepsilon_0]$, $\varepsilon_0 \ll R$ and $K_8$ depends on $\gamma$, $\mu$, $\delta$, $R$, $K_1$, $K_2$, $K_3$, $K_6$, but is independent of $\varepsilon$. If \eqref{3.29} holds then \eqref{3.31} is satisfied for $0<T<\lambda$, $Y\in[\delta,R]$, $\lambda$ is defined in \eqref{3.22*}.
   \end{lem}
   \begin{proof}
 Estimate \eqref{3.31} follows from the Schauder estimates for equation \eqref{3.1} and Lemmas 3.1 - 3.6.   
   \end{proof}
   \begin{theorem}\label{thm3}
  Suppose $\alpha> 0$, $\beta > 0$.\\
  
  (i) If $n > 0$, $\kappa\ne 0$ or $\kappa = 0$, $n \in \R$, then problem \eqref{1.1}-\eqref{1.4} has a unique classical solution $ U(T,Y) \in C^2(Q) \cap C^1(\overline{Q})$;\\
  
  (ii) If \eqref{3.28} holds, then problem \eqref{1.1}-\eqref{1.4} has a unique classical solution $ U(T,Y) \in C^2(Q_0) \cap C^1 (\overline{Q_0})$, where $Q_0=\left\lbrace (T,Y); T\in [0,T_0], Y\in [0,R] \right\rbrace$ for every  $T_0 < \infty$.\\
  
  (iii) If \eqref{3.29} holds, then problem \eqref{1.1}-\eqref{1.4} has a unique local classical solution $ U(T,Y) \in C^2(Q_\lambda) \cap C^1 (\overline{Q_\lambda})$, where $Q_\lambda=\left\lbrace (T,Y); 0<T<\lambda, 0<Y<R \right\rbrace$ and $\lambda$ is defined in \eqref{3.22*}.
       \end{theorem}
     \begin{proof}
 From Lemma 3.5 the equation \eqref{3.1} becomes uniformly parabolic in $\overline{Q}$ for case (i), in  $\overline{Q_0}$ for case (ii) and in $\overline{Q_\lambda}$ in case (iii), respectively. Existence of a classical $ C^4(Q) \cap C^2(\overline{Q})$, respectively, $ C^4(Q_0) \cap C^2(\overline{Q_0})$ or $ C^2(Q_\lambda) \cap C^1(\overline{Q_\lambda})$, solution to \eqref{3.1} follows by means of the method of continuity on parameter and the Schauder theory \cite{Ladyzhenskaya et al}. 
 
 From Lemma 3.7 the sequences $\lbrace U^\varepsilon(T,Y)\rbrace$, $\lbrace U^\varepsilon_Y(T,Y)\rbrace$, $\lbrace U^\varepsilon_T(T,Y)\rbrace$, $\lbrace U^\varepsilon_{TY}(T,Y)\rbrace$, $\lbrace U^\varepsilon_{YY}(T,Y)\rbrace$ for $\varepsilon\longrightarrow 0$ are equicontinuous and uniformly bounded for $T \ge 0$, $Y\in [\delta, R]$, $\delta > 0$ in case (i), for $T\in [0,T_0]$,  $T_0 < \infty$, $Y\in [\delta,R]$, $\delta> 0$ for case (ii) and for $T\in [0,\lambda], Y\in [\delta,R]$ in case (iii). Moreover, $\lbrace U^\varepsilon(T,Y)\rbrace$ and $\lbrace U^\varepsilon_Y(T,Y)\rbrace$ are equicontinuous and uniformly bounded in $(\overline{Q})$ in case (i), in $(\overline{Q_0})$ in case (ii) and in $\overline{Q_\lambda}$ in case (iii) with constants independent of $\varepsilon$. By means of the Arzela-Ascoli theorem and a diagonalization argument, there exists a subsequence $\lbrace U^{\varepsilon_i}(T,Y)\rbrace$, which converges to the desired solution for $\varepsilon_i\longrightarrow 0$        
     \end{proof}

\bigskip

\subsection*{Acknowledgments}
N.K. has been supported by the Grant
No BG05M2OP001-1.001-0003-C01, financed by the Science and Education for Smart Growth
Operational Program (2018-2023).

\subsection*{Conflict of interest}

The authors declare no potential conflict of interests.

\end{document}